\documentclass[12pt,draft]{article} 
\usepackage{amsmath,amssymb,amsthm,amsfonts, indentfirst}
\makeatletter
\def\@cite#1#2{[{{\bfseries #1}\if@tempswa , #2\fi}]}
\renewcommand{\section}{%
\@startsection{section}{1}{\z@}
{0.5truecm plus -1ex minus -.2ex}%
{1.0ex plus .2ex}{\bfseries\large}}
\def\@seccntformat#1{\csname the#1\endcsname.\ }
\makeatother

\numberwithin{equation}{section} 
\newtheorem{thm}{Theorem}[section]

\newtheorem{lemma}[thm]{Lemma}

\theoremstyle{definition}
\newtheorem{df}{Definition}[section]

\def\d{\delta}
\def\D{\Delta}
\def\om{\Omega}
\def\na{\nabla}
\def\vp{\varphi}
\def\e{\varepsilon}
\def\h{\hspace}

\def\etmax{T_{{\rm max},\e}}
\def\n#1{n_#1}
\def\en#1{n_{#1,\e}}
\def\norm#1{\|#1\|}

\begin{document}
\footnote[0]
    {2010{\it Mathematics Subject Classification}\/. 
    Primary: 35B40; Secondary: 35K55, 35Q30, 92C17.
%35B40: Asymptotic behavior of solutions
%35K55: Nonlinear parabolic equations
%35Q30: Navier-Stokes equations
%92C17: Cell movement (chemotaxis, etc.)
    }
%==========================title==========================
\begin{center}
    \Large{{\bf 
    Boundedness and stabilization\\ 
       in a three-dimensional two-species 
       chemotaxis-Navier--Stokes system\\ 
       with competitive kinetics
    }}
\end{center}
\vspace{5pt}
%===========================author=========================
\begin{center}
    Misaki Hirata, %\footnote[0]{
    %\\[2mm]
    Shunsuke Kurima, 
    %\\[2mm]
    Masaaki Mizukami\footnote{Partially supported by JSPS Research Fellowships 
       for Young Scientists, No.\ 17J00101.    
    }, 
    %\\[2mm]
    Tomomi Yokota%
   \footnote{Corresponding author}%
   \footnote{Partially supported by Grant-in-Aid for
    Scientific Research (C), No.\,16K05182.}
   \footnote[0]{
    E-mail: 
    {\tt misaki.hirata0928@gmail.com}, 
    {\tt shunsuke.kurima@gmail.com}, }\footnote[0]{
    {\tt masaaki.mizukami.math@gmail.com}, 
    {\tt yokota@rs.kagu.tus.ac.jp}%[2mm]
    }\\
    \vspace{12pt}
    Department of Mathematics, 
    Tokyo University of Science\\
    1-3, Kagurazaka, Shinjuku-ku, 
    Tokyo 162-8601, Japan\\
%    \vspace{12pt}
    \vspace{2pt}
\end{center}
\begin{center}    
    \small \today
\end{center}

\vspace{2pt}
%=====================  Abstract  =======================
\newenvironment{summary}
{\vspace{.5\baselineskip}\begin{list}{}{%
     \setlength{\baselineskip}{0.85\baselineskip}
     \setlength{\topsep}{0pt}
     \setlength{\leftmargin}{12mm}
     \setlength{\rightmargin}{12mm}
     \setlength{\listparindent}{0mm}
     \setlength{\itemindent}{\listparindent}
     \setlength{\parsep}{0pt}
     \item\relax}}{\end{list}\vspace{.5\baselineskip}}
\begin{summary}
{\footnotesize {\bf Abstract.}
    This paper is concerned with 
    the two-species chemotaxis-Navier--Stokes system 
    with Lotka--Volterra competitive kinetics
    %%%%%%%%%%%%%%%%%%%%%%%%%%%%%%%%%%%%%%%%%%%%%%%%%%%%%%%%
      	\begin{align*}
          \begin{cases}
            (\n1)_t+u\cdot\na\n1
            =\D\n1-\chi_1\na\cdot(\n1\na c)+\mu_1\n1(1-\n1-a_1\n2)
           &\text{in}\ \om\times(0,\infty),
          \\
            (\n2)_t+u\cdot\na\n2
            =\D\n2-\chi_2\na\cdot(\n2\na c)+\mu_2\n2(1-a_2\n1-\n2)
           &\text{in}\ \om\times(0,\infty),
          \\
            \h{6.3mm}c_t+u\cdot\na c
            =\D c-(\alpha\n1+\beta\n2)c
           &\text{in}\ \om\times(0,\infty),
          \\
            \h{3.1mm}u_t+(u\cdot\na)u
            =\D u+\nabla P+(\gamma\n1+\d\n2)\na\Phi,
            \quad\na\cdot u=0
           &\text{in}\ \om\times(0,\infty)
          \end{cases}
        \end{align*}
    %%%%%%%%%%%%%%%%%%%%%%%%%%%%%%%%%%%%%%%%%%%%%%%%%%%%%%%%
under homogeneous Neumann boundary conditions and 
initial conditions, where $\Omega$ is  
a bounded domain in $\mathbb{R}^3$ 
with smooth boundary. 
Recently, in the $2$-dimensional setting, 
global existence and stabilization of classical solutions to the above system 
were first established. 
However, the $3$-dimensional case has not been studied: 
Because of difficulties in the Navier--Stokes system, we can not expect 
existence of classical solutions to the above system. 
The purpose of this paper is to obtain global existence 
of weak solutions to the above system, 
and their eventual smoothness and stabilization. 
}
\end{summary}
\vspace{10pt}

\newpage
%%==============================================================%%
%%==============                                  ==============%%
%%======                      Section1                    ======%%
%%====                                                      ====%%
%%==                                                          ==%%
%%====                     Introduction Results           ====%%
%%======                                                  ======%%
%%==============                                  ==============%%
%%==============================================================%%

\section{Introduction}

%========================  Problem  =============================
This paper deals with the following two-species 
chemotaxis-Navier--Stokes system with 
Lotka--Volterra competitive kinetics:
%%%%%%%%%%%%%%%%%%%%%%%%%%%%  (P)  %%%%%%%%%%%%%%%%%%%%%%%%%%%%%%
  \begin{equation}\label{P}
    \begin{cases}
      (\n1)_t+u\cdot\na\n1
      =\D\n1-\chi_1\na\cdot(\n1\na c)+\mu_1\n1(1-\n1-a_1\n2)
     &\text{in}\ \om\times(0,\infty),
    \\[1mm]
      (\n2)_t+u\cdot\na\n2
      =\D\n2-\chi_2\na\cdot(\n2\na c)+\mu_2\n2(1-a_2\n1-\n2)
     &\text{in}\ \om\times(0,\infty),
    \\[1mm]
      \h{7.1mm}c_t+u\cdot\na c
      =\D c-(\alpha\n1+\beta\n2)c
     &\text{in}\ \om\times(0,\infty),
    \\[1mm]
      \h{1.3mm}u_t+\kappa(u\cdot\na)u 
      =\D u+\na P+(\gamma\n1+\d\n2)\na\Phi, 
      \quad\na\cdot u=0
     &\text{in}\ \om\times(0,\infty),
    \\[1mm]
      \h{3.9mm}\partial_\nu\n1=\partial_\nu\n2=\partial_\nu c=0,
      \quad u=0
     &\text{on}\ \partial\om\times(0,\infty),
    \\[1mm]
      \h{0.5mm}
      \n1(\cdot,0)=n_{1,0},\quad
      \n2(\cdot,0)=n_{2,0},\quad
      c(\cdot,0)=c_0,\quad
      u(\cdot,0)=u_0
     &\text{in}\ \om,
    \end{cases}
  \end{equation}
%%%%%%%%%%%%%%%%%%%%%%%%%%%%%%%%%%%%%%%%%%%%%%%%%%%%%%%%%%%%%
where $\om$ is a bounded domain 
in $\mathbb{R}^3$ with smooth boundary $\partial\om$ and 
$\partial_\nu$ denotes differentiation with respect to the 
outward normal of $\partial\om$; 
$\kappa=1$, $\chi_1,\chi_2,a_1,a_2\geq0$ and 
$\mu_1,\mu_2,\alpha,\beta,\gamma,\d>0$ are 
constants; $n_{1,0},n_{2,0},c_0,u_0,\Phi$ 
are known functions satisfying
  \begin{align}
     \label{condi;ini1}
    &0<n_{1,0},n_{2,0}\in C(\overline{\om}), 
     \quad
     0<c_0\in W^{1,q}(\om),
     \quad
     u_0\in D(A^{\theta}),
    \\
     \label{condi;ini2}
    &\Phi\in C^{1+\lambda}(\overline{\om})
  \end{align}
for some $q>3$, $\theta\in(\frac{3}{4},1)$, $\lambda \in (0,1)$
and $A$ denotes the realization of the Stokes operator 
under homogeneous Dirichlet boundary conditions 
in the solenoidal subspace $L^2_\sigma (\Omega)$ of $L^2(\om)$.

%
%==========================  Back graund  ===========================
%
In the mathematical point of view,  
difficulties of this problem are mainly caused by the chemotaxis terms 
$-\chi_1 \nabla \cdot (n_1\nabla c)$, $-\chi_2 \nabla \cdot (n_2\nabla c)$, 
the competitive kinetics 
$\mu_1 n_1(1-n_1-a_1n_2)$, $\mu_2 n_2(1-a_2n_1 -n_2)$ 
and the Navier--Stokes equation 
which is the fourth equation in 
\eqref{P}. In the case that $\n2=0$, 
global existence of weak solutions, and their eventual smoothness 
and stabilization were shown in \cite{Johannes(2016)}. 
On the other hand, in the case that $n_2 \neq 0$ 
and $\Omega \subset \mathbb{R}^2$, 
global existence and boundedness of classical solutions to 
\eqref{P} have been attained (\cite{HKMY(2017)}). 
%Moreover \tcb{(Who?)} studied 
%convergence rate of the classical solutions to \eqref{P}. 
Moreover, 
in the case that $\kappa = 0$ in \eqref{P}, 
which namely means that 
the fourth equation in \eqref{P} 
is the {\it Stokes} equation, 
global existence and stabilization 
%in the chemotaxis-Stokes system 
can be found in \cite{CKM1}; 
in the case that $\kappa = 0$ in \eqref{P} and that 
$-(\alpha n_1 + \beta n_2)c$ is replaced 
with $+ \alpha n_1 + \beta n_2 - c$, 
global existence and boundedness of classical solutions to 
the Keller--Segel-Stokes system 
and their asymptotic behaviour are found in \cite{CKM2}. 
%problems which related to a 3-dimensional 
%{\it chemotaxis-Stokes} system and a 3-dimensional 
%{\it Keller--Segel-Stokes} system have also been studied; 
%global existence and stabilization in the chemotaxis-Stokes system 
%can be found in \cite{CKM1}; 
%existence of global classical bounded solutions to 
%the Keller--Segel-Stokes system 
%and their asymptotic behaviour are in \cite{CKM2}. 

As we mentioned above, 
global classical solutions are found 
in \eqref{P} in the 2D setting 
and the case that $\kappa = 0$. 
However, global existence of solutions in 3-dimensional setting 
has not been attained. 
Thus the main purposes of this paper is 
to obtain global existence of solutions to 
\eqref{P} in the case that $\Omega \subset \mathbb{R}^3$.  
Nevertheless, because of the difficulties of 
the Navier--Stokes equation, we can not expect global existence 
of {\it classical solutions} to \eqref{P} in the 3-dimensional case. 
Therefore our goal is to obtain global existence 
of {\it weak solutions} to \eqref{P} 
in the following sense. 
%and their large-time behaviour 
%in the 3-dimensional case. 

%
%============================  Purpose  =============================
%
%
%First, we will define a weak solution in order to state main theorems.
%%%%%%%%%%%%%%%%%%%%%  Def.1.1 (Weak Solution)  %%%%%%%%%%%%%%%%%%%%%

\smallskip

\begin{df}\label{def;weaksol}
  A quadruple $(\n1,\n2,c,u)$ is called a $($global$)$ weak solution
  of \eqref{P} if
  \begin{align*}
    \n1, \n2
    &\in L^2_{\rm loc}([0,\infty);L^2(\om))
     \cap L^{\frac{4}{3}}_{\rm loc}([0,\infty);W^{1,\frac{4}{3}}(\om)),
   \\
    c
    &\in L^2_{\rm loc}([0,\infty);W^{1,2}(\om)),
   \\
    u
    &\in L^2_{\rm loc}([0,\infty);W^{1,2}_{0,\sigma}(\om))
  \end{align*}
  and for all $T>0$ the identities
  \begin{align*}
%------------------------  \n1  ----------------------------
    &-\int^\infty_0\!\!\!\!\int_\om\n1\vp_t
     -\int_\om n_{1,0}\vp(\cdot,0)
     -\int^\infty_0\!\!\!\!\int_\om\n1u\cdot\na\vp
   \\
    &\h{7.0mm}=-\int^\infty_0\!\!\!\!\int_\om\na\n1\cdot\na\vp
      +\chi_1\int^\infty_0\!\!\!\!\int_\om\n1\na c\cdot\na\vp
      +\mu_1\int^\infty_0\!\!\!\!\int_\om\n1(1-\n1-a_1\n2)\vp,
   \\[2.0mm]
%------------------------  \n2  ----------------------------
    &-\int^\infty_0\!\!\!\!\int_\om\n2\vp_t
     -\int_\om n_{2,0}\vp(\cdot,0)
     -\int^\infty_0\!\!\!\!\int_\om\n2u\cdot\na\vp
   \\
    &\h{7.0mm}=-\int^\infty_0\!\!\!\!\int_\om\na\n2\cdot\na\vp
      +\chi_2\int^\infty_0\!\!\!\!\int_\om\n2\na c\cdot\na\vp
      +\mu_2\int^\infty_0\!\!\!\!\int_\om\n2(1-a_2\n1-\n2)\vp,
   \\[2.0mm]
%-------------------------  c  -----------------------------
    &-\int^\infty_0\!\!\!\!\int_\om c\vp_t
     -\int_\om c_0\vp(\cdot,0)
     -\int^\infty_0\!\!\!\!\int_\om cu\cdot\na\vp
   \\
    &\h{7.0mm}=-\int^\infty_0\!\!\!\!\int_\om\na c\cdot\na\vp
      -\int^\infty_0\!\!\!\!\int_\om(\alpha\n1+\beta\n2)c\vp,
   \\[2.0mm]
%-------------------------  u  -----------------------------
    &-\int^\infty_0\!\!\!\!\int_\om u\cdot\psi_t
     -\int_\om u_0\cdot\psi(\cdot,0)
     -\int^\infty_0\!\!\!\!\int_\om u\otimes u\cdot\na\psi
   \\
    &\h{7.0mm}=-\int^\infty_0\!\!\!\!\int_\om\na u\cdot\na\psi
     +\int^\infty_0\!\!\!\!\int_\om(\gamma\n1+\d\n2)\na\psi\cdot\na\Phi
  \end{align*}
  hold for all $\vp\in C^{\infty}_0(\overline{\om}\times [0,\infty))$
  and all $\psi\in C^{\infty}_{0,\sigma}(\om \times [0,\infty))$,
  respectively.
\end{df}

Now the main results read as follows. 
The first theorem is concerned with 
global existence of weak solutions to \eqref{P}. 

\smallskip

% 
%%%%%%%%%%%%%%%%%%%%%%%%%  Theorem1.2  %%%%%%%%%%%%%%%%%%%%%%%%%%
\begin{thm}\label{mainthm1}
  Let $\om\subset\mathbb{R}^3$ be a bounded smooth domain 
  and let $\chi_1,\chi_2,a_1,a_2\geq 0$ and 
  $\mu_1,\mu_2,\alpha,\beta,\gamma,\d>0$.
  Assume that $n_{1,0},n_{2,0},c_0,u_0$ satisfy \eqref{condi;ini1}
  with some $q>3$ and $\theta \in(\frac{3}{4},1)$ 
  and $\Phi\in C^{1+\lambda}(\overline{\om})$ 
  for some $\lambda\in (0,1)$.
  Then there is a weak solution of \eqref{P},
  which can be approximated by a sequence of solutions 
  $(\en1,\en2,c_\e,u_\e)$ of 
  \eqref{Pe} $($see Section $\ref{section2})$ 
  in a pointwise manner.
\end{thm}

\smallskip

The second theorem gives eventual smoothness and stabilization. 

\smallskip

%%%%%%%%%%%%%%%%%%%%%%%%%  Theorem1.3  %%%%%%%%%%%%%%%%%%%%%%%%%%
\begin{thm}\label{mainthm2}
  Let the assumption of Theorem $\ref{mainthm1}$ be satisfied. 
  Then there are $T>0$ and $\alpha' \in(0,1)$ such that the solution
  $(\n1,\n2,c,u)$ given by Theorem $\ref{mainthm1}$ satisfies
  \begin{align*}
    \n1,\n2,c
    \in C^{2+\alpha',1+\frac{\alpha'}{2}}(\overline{\om}\times[T,\infty)),
    \h{3.0mm}
    u
    \in C^{2+\alpha',1+\frac{\alpha'}{2}}(\overline{\om}\times[T,\infty)).
  \end{align*}
  Moreover, the solution of $(\ref{P})$ 
  has the following properties\/{\rm :}
  \begin{itemize}
  \item[{\rm(i)}]
    Assume that $a_1,a_2\in (0,1)$. Then 
    \begin{align*}
      \n1(\cdot,t)\to N_1,      \h{3.0mm}
      \n2(\cdot,t)\to N_2,      \h{3.0mm}
      c(\cdot,t)\to0,          \h{3.0mm}
      u(\cdot,t)\to0           \h{3.0mm}
      \mbox{in}\ L^\infty(\om) %\h{3.0mm}
    \end{align*}
    as $t\to\infty$, where 
    \begin{align*}
    N_1:=\frac{1-a_1}{1-a_1a_2}, \h{3.0mm}
    N_2:=\frac{1-a_2}{1-a_1a_2}.
    \end{align*}
  \item[{\rm(ii)}]
    Assume that $a_1\geq 1 > a_2$. Then
    \begin{align*}
      \n1(\cdot,t)\to 0,       \h{3.0mm}
      \n2(\cdot,t)\to 1,       \h{3.0mm}
      c(\cdot,t)\to 0,         \h{3.0mm}
      u(\cdot,t)\to 0          \h{3.0mm}
      \mbox{in}\ L^\infty(\om) \h{3.0mm}
%      \mbox{as}\ t\to \infty.
    \end{align*}
   as $t\to \infty$. 
  \end{itemize}
\end{thm}

\smallskip

The proofs of the main theorems are based on the arguments 
in \cite{Johannes(2016)}.  
The strategies for the proofs is to 
construct energy estimates for the solution 
$(\en1,\en2,c_\e,u_\e)$ of \eqref{Pe}. 
In Section \ref{section2} we consider the energy function $\cal F_\e$ defined as
\begin{align*}
   {\cal F}_\e:=
      \int_\om\en1\log\en1
      +\int_\om\en2\log\en2
      +\frac{\chi}{2}\int_\om\frac{|\na c_\e |^2}{c_\e }
      +k_4\chi\int_\om|u_\e |^2
\end{align*}
with some constant $\chi>0$. 
Noting that 
for all $\rho,\xi_i>0$ there exists $C>0$ such that
  \begin{align*}
      &\int_\om\na c_\e\cdot\na\en{i}
        \bigg(
          \frac{\chi_i}{1+\e\en{i}}
         -\frac{\chi\alpha\ (\mbox{or}\ \chi\beta)}{1+\e(\alpha\en1+\beta\en2)}
        \bigg)
    \\
      &\leq
       \rho\int_\om\frac{|\na c_\e|^4}{c_\e^3}
      +\xi_i\int_\om\frac{|\na\en{i}|^2}{\en{i}}
      +C\int_\om\en{i}^2
    \quad (i=1,2),
  \end{align*}
  which did not appear in the previous work \cite{Johannes(2016)}, 
from the estimate for the energy function $\cal F_\e$ 
we obtain global-in-time solvability of approximate solutions.  
Then we moreover see convergence as $\e \searrow 0$. 
Furthermore, in Section \ref{section3}, 
according to an argument similar to \cite{HKMY(2017)},
by putting
  \begin{align*}
   {\cal G}_{\e,B}:=\int_\om\left(\en1-N_1\log\frac{\en1}{N_1}\right)
                    +\int_\om\left(\en2-N_2\log\frac{\en2}{N_2}\right)
                    +\frac{B}{2}\int_\om c_\e^2
  \end{align*}
with suitable constant $B>0$ 
and establishing the H\"{o}lder estimates for 
the solution of \eqref{P} through the estimate for 
the energy function 
${\cal G}_{\e,B}$, we can discuss convergence of 
$\big(\n1(\cdot,t),\n2(\cdot,t),c(\cdot,t),u(\cdot,t)\big)$ 
as $t\to\infty$.
%
%
%%==============================================================%%
%%==============                                  ==============%%
%%======                      Section2                    ======%%
%%====                                                      ====%%
%%==                                                          ==%%
%%====                  Proof of Theorem1.2                 ====%%
%%======                                                  ======%%
%%==============                                  ==============%%
%%==============================================================%%

\section{Proof of Theorem 1.2 (Global existence)}\label{section2}
%%%%%%%%%%%%%%%%%%%%%%%%%%%%  (Pe)  %%%%%%%%%%%%%%%%%%%%%%%%%%%%%%
We will start by considering an approximate problem with parameter $\e >0$, 
namely:
\begin{equation}\label{Pe}
  \begin{cases}
     (\en1)_t+u_\e\cdot\na\en1
     =\D\en1
     -\chi_1\na\cdot\Big(\frac{\en1}{1+\e\en1}\na c_\e\Big)
     +\mu_1\en1(1-\en1-a_1\en2),
  \\[2mm]
     (\en2)_t+u_\e\cdot\na\en2
     =\D\en2
     -\chi_2\na\cdot\Big(\frac{\en2}{1+\e\en2}\na c_\e\Big)
     +\mu_2\en2(1-a_2\en1-\en2),
  \\[2mm]
     (c_\e)_t+u_\e\cdot\na c_\e
     =\D c_\e
     -c_\e\frac{1}{\e}\log\big(1+\e(\alpha\en1+\beta\en2)\big),
  \\[2mm]
     (u_\e)_t+(Y_\e u_\e\cdot\na)u_\e
     =\D u_\e
     +\na P_\e
     +(\gamma\en1+\d\en2)\na\Phi,
     \quad\na\cdot u_\e=0,
  \\[2mm]
     \partial_\nu\en1|_{\partial\om}
     =\partial_\nu\en2|_{\partial\om}
     =\partial_\nu c_\e|_{\partial\om}=0,
     \quad u_\e|_{\partial\om}=0,
  \\[2mm]
     \en1(\cdot,0)=n_{1,0},\quad
     \en2(\cdot,0)=n_{2,0},\quad
     c_\e(\cdot,0)=c_0,\quad
     u_\e(\cdot,0)=u_0,
  \end{cases}
\end{equation}
%%%%%%%%%%%%%%%%%%%%%%%%%%%%%%%%%%%%%%%%%%%%%%%%%%%%%%%%%%%%%%%%%
where $Y_\e=(1+\e A)^{-1}$,
and provide estimates for its solutions. 
We first give the following result which states 
local existence in \eqref{P}. 

\smallskip

%
%%%%%%%%%%%%%%%%%%%%%%%%%%  Lemma 2.1  %%%%%%%%%%%%%%%%%%%%%%%%%%
\begin{lemma}
  Let $\chi_1,\chi_2,a_1,a_2\geq0$, 
  $\mu_1,\mu_2,\alpha,\beta,\gamma,\d>0$, and
  $\Phi\in C^{1+\lambda}(\overline{\om})$ for some $\lambda \in (0,1)$ 
  and  assume that $n_{1,0},n_{2,0},c_0,u_0$ satisfy \eqref{condi;ini1} 
  with some $q>3,\theta \in(\frac{3}{4},1)$.  
  Then for all $\e > 0$ there are $\etmax$ and 
  uniquely determined functions\/{\rm :}
    \begin{align*}
       \en1,\en2
      &\in C^0(\overline{\om}\times[0,\etmax))
       \cap C^{2,1}(\overline{\om}\times(0,\etmax)),
    \\
       c_\e
      &\in C^0(\overline{\om}\times[0,\etmax))
       \cap C^{2,1}(\overline{\om}\times(0,\etmax))
       \cap L^\infty_{\rm loc}([0,\etmax);W^{1,q}(\om)),
    \\
       u_\e
      &\in C^0(\overline{\om}\times[0,\etmax)) 
       \cap C^{2,1}(\overline{\om}\times(0,\etmax)),
    \end{align*}
  which together with some 
  $P_\e\in C^{1,0}(\overline{\om}\times(0,\etmax))$ 
  solve \eqref{Pe} classically. 
  Moreover, $\en1$, $\en2$ and $c_\e$ are positive 
  and the following alternative holds\/{\rm :} 
  $\etmax=\infty$ or
  \begin{align}\label{extension}
    %\notag
     \norm{n_{1,\e}(\cdot,t)}_{L^{\infty}(\om)}
    +\norm{n_{2,\e}(\cdot,t)}_{L^{\infty}(\om)}
    +\norm{c_\e(\cdot,t)}_{W^{1,q}(\om)}
    +\norm{A^\theta u_\e(\cdot,t)}_{L^2(\om)}
    \to  \infty
  \end{align}
as $t\nearrow \etmax$.
\end{lemma}

\smallskip

We next show the following lemma which holds a key 
for the proof of Theorem \ref{mainthm1}. 
This lemma derives the estimate for the energy function. 

\smallskip

%%%%%%%%%%%%%%%%%%%%%%%%%%  Lemma 2.2  %%%%%%%%%%%%%%%%%%%%%%%%%%
\begin{lemma}\label{lem:2.10}
For all $\xi_1,\xi_2\in(0,1)$ and $\chi>0$ there are 
$C,\overline{C},\widetilde{C},k,\overline{k}>0$ such that
  \begin{align*}
     {\cal F}_\e:=
      \int_\om\en1\log\en1
     +\int_\om\en2\log\en2
     +\frac{\chi}{2}\int_\om\frac{|\na c_\e|^2}{c_\e}
     +\overline{k}\chi\int_\om|u_\e|^2
  \end{align*}
satisfies
  \begin{align*}
   \frac{d}{dt}{\cal F}_\e
   \leq
   &-\frac{\mu_1}{4}\int_\om\en1^2\log\en1
    -\frac{\mu_2}{4}\int_\om\en2^2\log\en2
   \\
   &-(1-\xi_1)\int_\om\frac{|\na\en1|^2}{\en1}
    -(1-\xi_2)\int_\om\frac{|\na\en2|^2}{\en2}
    +C\int_\om\en1^2
    +\overline{C}\int_\om\en2^2
    +\widetilde{C}
   \\
   &-k\int_\om c_\e|D^2\log c_\e|^2
    -k\int_\om\frac{|\na c_\e|^4}{c^3_\e}
    -k\int_\om|\na u_\e|^2 
  \end{align*}
on $(0,\etmax)$ for all $\e>0$.
\end{lemma}
%
%%%%%%%%%%%%%%%%%%%%%  Proof of Lemma 2.2  %%%%%%%%%%%%%%%%%%%%%%
\begin{proof}
%
%%============================  n_1  ============================
%
Noting, the boundedness of $s(1-s)$ and $s(1-\frac{s}{2})\log s$,
we have that there exists $C_1>0$ such that
  \begin{align}
    \notag
      &\frac{d}{dt}\int_\om\en1\log\en1
    \\[2mm]
    \notag
      &=-\int_\om\frac{|\na\en1|^2}{\en1}
        +\chi_1\int_\om\frac{\na c_\e\cdot\na\en1}{1+\e\en1}
    \\
    \notag
      &\quad\,
       +\mu_1 \int_\om\en1(1-\en1-a_1\en2)\log\en1
       +\mu_1\int_\om\en1(1-\en1-a_1\en2)
    \\[2mm]
    \notag
      &\leq
       -\int_\om\frac{|\na\en1|^2}{\en1}
       +\chi_1\int_\om\frac{\na c_\e\cdot\na\en1}{1+\e\en1}
       -\frac{\mu_1}{2}\int_\om\en1^2\log\en1
      \\
      \label{n1logn1}
      &\quad\,
       -\mu_1 a_1 \int_\om\en1\en2\log\en1
       -\mu_1 a_1 \int_\om\en1\en2
       +C_1.
  \end{align}
%
%%============================  n_2  ============================
%
Similarly, there is $C_2>0$ such that
  \begin{align}
    \notag
      \frac{d}{dt}\int_\om\!\!\en2\log\en2
      &\leq
       -\int_\om\frac{|\na\en2|^2}{\en2}
       +\chi_2\int_\om\frac{\na c_\e\cdot\na\en2}{1+\e\en2}
       -\frac{\mu_2}{2}\int_\om\en2^2\log\en2
    \\
    \label{n2logn2}
      &\quad\,
       -\mu_2 a_2 \int_\om\en1\en2\log\en2
       -\mu_2 a_2 \int_\om\en1\en2
       +C_2.
  \end{align}
%
%%=============================  c  =============================
%
According to an argument similar to that in the proof of 
\cite[Lemma 2.8]{Johannes(2016)}, there exist 
$k_1,C_3,C_4>0$ such that
  \begin{align}
    \notag
      \frac{d}{dt}\int_\om\frac{|\na c_\e|^2}{c_\e}
      \leq
      &-k_1\int_\om c_\e|D^2\log c_\e|^2
       -k_1\int_\om\frac{|\na c_\e|^4}{c_\e^3}
    \\
    \label{|nablac|^2/c}
      &
       +C_3
       +C_4\int_\om|\na u_\e|^2
       -2\int_\om
         \frac{\alpha\na c_\e\cdot\na\en1+\beta\na c_\e\cdot\na\en2}
              {1+\e(\alpha\en1+\beta\en2)}.
  \end{align}
%
%%=============================  u  =============================
%
Now we let $\overline{k},\eta_1,\eta_2,k$ be constants satisfying
$\frac{C_4}{2}-\overline{k}=-\frac{k_1}{4}$,
$\eta_1=\frac{\mu_1}{4\overline{k}\chi}$,
$\eta_2=\frac{b\mu_2}{4\overline{k}\chi}$ 
and $k=\frac{\chi k_1}{4}$.
Then we have
  \begin{align*}
  \frac{d}{dt}\int_\om|u_\e|^2
     =&-2\int_\om|\na u_\e|^2
       -2\int_\om u_\e\cdot(Y_\e u_\e\cdot\na)u_\e
       +2\int_\om u_\e\cdot(\gamma\en1+\d\en2)\na\Phi.
  \end{align*}
From the Schwarz inequality, the Poincar\'{e} inequality, 
the Young inequality and the fact that
$\int_\om\vp^2\leq a\int_\om\vp^2\log\vp+|\om|e^\frac{1}{a}$
holds for any positive function $\vp$ and any $a>0$,
there exist $C_5,C_{\eta_1},C_{\eta_2}>0$ such that
  \begin{align*}
    \gamma\int_\om|\en1\na\Phi\cdot u_\e|
      &\leq\gamma\norm{\na\Phi}_{L^\infty}
           \left(\int_\om \en1^2\right)^\frac{1}{2}
           \left(\int_\om|u_\e|^2\right)^\frac{1}{2}
    \\
      &\leq\gamma\norm{\na\Phi}_{L^\infty}
           \left(\int_\om \en1^2\right)^\frac{1}{2}
           \left(C_5\int_\om|\na u_\e|^2\right)^\frac{1}{2}
    \\
      &\leq\gamma^2 C_5\norm{\na\Phi}_{L^\infty}^2\int_\om \en1^2
          +\frac{1}{4}\int_\om|\na u_\e|^2
    \\
      &\leq\frac{\eta_1}{2}\int_\om\en1^2\log\en1
          +\frac{C_{\eta_1}}{2}
          +\frac{1}{4}\int_\om|\na u_\e|^2
  \end{align*}
and
  \begin{align*}
    \d\int_\om|\en2\na\Phi\cdot u_\e|
    \leq\frac{\eta_2}{2}\int_\om\en2^2\log\en2
        +\frac{C_{\eta_2}}{2}
        +\frac{1}{4}\int_\om|\na u_\e|^2
  \end{align*}
hold.
Therefore we have 
  \begin{align}\label{|u|^2}
    \frac{d}{dt}\int_\om|u_\e|^2
      \leq-\int_\om|\na u_\e|^2
          +\eta_1\int_\om\en1^2\log\en1
          +\eta_2\int_\om\en2^2\log\en2
          +C_{\eta_1}
          +C_{\eta_2}.
  \end{align}
%
%%============================  all  ============================
%
Thus a combination of \eqref{n1logn1}, \eqref{n2logn2}, 
\eqref{|nablac|^2/c} and \eqref{|u|^2} leads to
  \begin{align*}
    %\notag
      &\frac{d}{dt}
       \bigg[
         \int_\om\en1\log\en1
        +\int_\om\en2\log\en2
        +\frac{\chi}{2}\int_\om\frac{|\na c_\e|^2}{c_\e}
        +\overline{k}\chi\int_\om|u_\e|^2
       \bigg]
    \\[2mm]
    %\notag
      &\leq
       \Big(\overline{k}\chi\eta_1-\frac{\mu_1}{2}\Big)
       \int_\om\en1^2\log\en1
      +\Big(\overline{k}\chi\eta_2-\frac{\mu_2}{2}\Big)
       \int_\om\en2^2\log\en2
    \\
    %\notag
      &\quad\,
       -\bigg(
          \int_\om\frac{|\na\en1|^2}{\en1}
          +\int_\om\frac{|\na\en2|^2}{\en2}
        \bigg)
       +\Big(\frac{\chi}{2}C_4-\overline{k}\chi\Big)
        \int_\om|\na u_\e|^2
    \\
    %\notag
      &\quad\,
       +\int_\om\na c_\e\cdot\na\en1
        \bigg(
          \frac{\chi_1}{1+\e\en1}
          -\frac{\chi\alpha}{1+\e(\alpha\en1+\beta\en2)}
        \bigg)
    \\
    %\notag
      &\quad\,
       +\int_\om\na c_\e\cdot\na\en2
        \bigg(
          \frac{\chi_2}{1+\e\en2}
          -\frac{\chi\beta}{1+\e(\alpha\en1+\beta\en2)}
        \bigg)
    \\
    %\notag
      &\quad\,
       -\frac{\chi}{2}k_1\int_\om c_\e|D^2\log c_\e|^2
       -\frac{\chi}{2}k_1\int_\om\frac{|\na c_\e|^4}{c_\e^3}
       +C_1
       +C_2
       +\frac{\chi}{2}C_3
       +\overline{k}\chi(C_{\eta_1}+C_{\eta_2})
    \\
      &\quad\,
       -\mu_1a_1\int_\om\en1\en2(\log\en1+1)
       -\mu_2a_2\int_\om\en1\en2(\log\en2+1).
  \end{align*}
Here, since  $\en1,\en2$ are nonnegative, we can find $C_6,C_7>0$ such that
  \begin{align*}
      &\int_\om \na c_\e\cdot\na\en1 
        \bigg(
          \frac{\chi_1}{1+\e\en1}
          -\frac{\chi\alpha}{1+\e(\alpha\en1+\beta\en2)}
        \bigg)
    \\[2mm]
      &\leq
       (\chi_1+\chi\alpha)\int_\om|\na c_\e\cdot\na\en1|
    \\[2mm]
      &\leq
       \frac{\chi k_1}{8\norm{c_0}^3_{L^\infty}}\int_\om|\na c_\e|^4
       +C_6\int_\om|\na\en1|^\frac{4}{3}
    \\[2mm]
      &\leq
       \frac{\chi k_1}{8}\int_\om\frac{|\na c_\e|^4}{c_\e^3}
       +\xi_1\int_\om\frac{|\na\en1|^2}{\en1}
       +C_7\int_\om\en1^2
  \end{align*}
and there is $C_8>0$ such that
  \begin{align*}
      &\int_\om\na c_\e\cdot\na\en2
        \bigg(
          \frac{\chi_2}{1+\e\en2}
          -\frac{\chi\beta}{1+\e(\alpha\en1+\beta\en2)}
        \bigg)
    \\
      &\leq
       \frac{\chi k_1}{8}\int_\om\frac{|\na c_\e|^4}{c_\e^3}
       +\xi_2\int_\om\frac{|\na\en2|^2}{\en2}
       +C_8\int_\om\en2^2,
  \end{align*}
which with the fact that $s\log s\geq-\frac{1}{e}$ $(s>0)$ 
enables us to obtain
  \begin{align*}
    &\Big(\overline{k}\chi\eta_1-\frac{\mu_1}{2}\Big)
     \int_\om\en1^2\log\en1
     +\Big(\overline{k}\chi\eta_2-\frac{\mu_2}{2}\Big)
      \int_\om\en2^2\log\en2
  \\
    &-\bigg(
        \int_\om\frac{|\na\en1|^2}{\en1}
        +\int_\om\frac{|\na\en2|^2}{\en2}
      \bigg)
     +\Big(\frac{\chi}{2}C_4-\overline{k}\chi\Big)
      \int_\om|\na u_\e|^2
  \\
    &+\int_\om\na c_\e\cdot\na\en1
      \bigg(
        \frac{\chi_1}{1+\e\en1}
        -\frac{\chi\alpha}{1+\e(\alpha\en1+\beta\en2)}
      \bigg)
  \\
    &+\int_\om\na c_\e\cdot\na\en2
      \bigg(
        \frac{\chi_2}{1+\e\en2}
        -\frac{\chi\beta}{1+\e(\alpha\en1+\beta\en2)}
      \bigg)
  \\
    &-\frac{\chi}{2}k_1\int_\om c_\e|D^2\log c_\e|^2
     -\frac{\chi}{2}k_1\int_\om\frac{|\na c_\e|^4}{c_\e^3}
     +C_1
     +C_2
     +\frac{\chi}{2}C_3
     +\overline{k}\chi(C_{\eta_1}+C_{\eta_2})
  \\
    &-\mu_1a_1\int_\om\en1\en2(\log\en1+1)
     -\mu_2a_2\int_\om\en1\en2(\log\en2+1)
  \\[2mm]
    &\leq
     -\frac{\mu_1}{4}\int_\om\en1^2\log\en1
     -\frac{\mu_2}{4}\int_\om\en2^2\log\en2
  \\
    &\quad
     -(1-\xi_1)\int_\om\frac{|\na\en1|^2}{\en1}
     -(1-\xi_2)\int_\om\frac{|\na\en2|^2}{\en1}
  \\
    &\quad
     -k\int_\om|\na u_\e|^2
     -k\int_\om c_\e|D^2\log c_\e|
     -k\int_\om\frac{|\na c_\e|^4}{c_\e^3}
     +C_7\int_\om\en1^2
     +C_8\int_\om\en2^2
     +C_9.
  \end{align*}
Therefore we obtain this lemma.
\end{proof}

%%%%%%%%%%%%%%%%%%%%%%%%%  Tmax,ε= ∞  %%%%%%%%%%%%%%%%%%%%%%%%%
{\it Proof of Theorem 1.2.}
Let $\tau=\min\{1,\frac{1}{2}\etmax\}$, $\xi_1,\xi_2\in (0,1)$ and $\chi>0$.
Lemma \ref{lem:2.10}, 
the facts that 
$s^2\log s\geq s\log s-\frac{1}{2e}$ $(s>0)$ 
and $\en1,\en2,c_\e>0$ imply 
  \begin{align*}
   \frac{d}{dt}{\cal F}_\e+{\cal F}_\e 
   \leq C\int_\om\en1^2 
        +\overline{C}\int_\om\en2^2 
        +\widetilde{C}'
  \end{align*}
for some $C,\overline{C},\widetilde{C}'>0$. 
According to \cite[Lemma 2.5]{Johannes(2016)}, 
there exists $C_1>0$ such that 
  \begin{align*}
    \int^{t+\tau}_t\!\!\!\!\int_\om\en{i}^2\leq C_1 
  \end{align*}
for all $t\in (0,\etmax -\tau)$ 
and each $i=1,2$. 
From the uniform Gronwall type lemma
(see e.g., \cite[Lemma 3.2]{TW}) 
we can find $C_2>0$ such that
  \begin{align}\label{tool1}
     \int_\om\en1\log\en1
     +\int_\om\en2\log\en2
     +\frac{\chi}{2}\int_\om\frac{|\na c_\e|^2}{c_\e}
     +\overline{k}\chi\int_\om|u_\e|^2
   \leq C_2
  \end{align}
for all $t\in(0,\etmax)$.
Moreover, we have from integration of the differential inequality in Lemma \ref{lem:2.10} 
over $(t,t+\tau)$ that for all $\xi_1,\xi_2\in(0,1)$ there is $C_3>0$ such that 
  \begin{align}
    \notag
      &\frac{\mu_1}{4}\int^{t+\tau}_t\!\!\!\!\int_\om\en1^2\log\en1
       +\frac{\mu_2}{4}\int^{t+\tau}_t\!\!\!\!\int_\om\en2^2\log\en2
       +k\int^{t+\tau}_t\!\!\!\!\int_\om c_\e|D^2\log c_\e|^2
    \\
    \label{tool2}
      &\quad
       +(1-\xi_1)\int^{t+\tau}_t\!\!\!\!\int_\om
                 \frac{|\na\en1|^2}{\en1}
       +(1-\xi_2)\int^{t+\tau}_t\!\!\!\!\int_\om
                  \frac{|\na\en2|^2}{\en2}
       \leq C_3
%    \\[2mm]
  \end{align}
  and
  \begin{align}
    \label{tool3}
      &\int^{t+\tau}_t\!\!\!\!\int_\om\frac{|\na c_\e|^4}{c_\e^3}
       +\int^{t+\tau}_t\!\!\!\!\int_\om|\na u_\e|^2
       \leq C_3
%    \\[2mm]
  \end{align}
  as well as
  \begin{align}
    \notag
      &\int^{t+\tau}_t\!\!\!\!\int_\om|\na\en1|^\frac{4}{3}
       +\int^{t+\tau}_t\!\!\!\!\int_\om|\na\en2|^\frac{4}{3}
    \\
    \label{tool4}
      &\quad
       +\int_\om|\na c_\e|^2
       +\int^{t+\tau}_t\!\!\!\!\int_\om|\na c_\e|^4
       +\int^{t+\tau}_t\!\!\!\!\int_\om\en1^2
       +\int^{t+\tau}_t\!\!\!\!\int_\om\en2^2
       \leq C_3
  \end{align}
for all $t\in[0,\etmax-\tau)$. 
Now we assume $\etmax<\infty$ for some $\e >0$. 
From \eqref{tool1}, \eqref{tool2}, \eqref{tool3} and \eqref{tool4},
we can see that there exists $C_4>0$ such that
  $$
  \begin{array}{rcl}
     \norm{n_{1,\e}(\cdot,t)}_{L^{\infty}(\om)}\leq C_4,
     &&
     \norm{n_{2,\e}(\cdot,t)}_{L^{\infty}(\om)}\leq C_4,\\
     \norm{c_\e(\cdot,t)}_{W^{1,q}(\om)}\leq C_4,
     &&
     \norm{A^\sigma u_\e(\cdot,t)}_{L^2(\om)}\leq C_4
  \end{array}
  $$
for all $t\in(0,\etmax)$, 
which is inconsistent with \eqref{extension}. 
Therefore we obtain $\etmax=\infty$ for all $\e>0$, which means 
global existence and 
boundedness of $(\en1,\en2,c_\e,u_\e)$.
We next verify convergence of the solution $(\en1,\en2,c_\e,u_\e)$.
Due to Lemma \ref{lem:2.10} and arguments similar to 
those in \cite{Johannes(2016)},
we establish that for all $T>0$ there is $C_5>0$ such that
  \begin{align}
    \notag
    \norm{(\en1)_t}_{L^1{((0,T);(W^{2,4}_0(\om))^*)}}\leq C_5,
    \qquad
    &
    \norm{(\en2)_t}_{L^1{((0,T);(W^{2,4}_0(\om))^*)}}\leq C_5,
    \\
    \label{lem:2.4}
    \norm{(c_\e)_t}_{L^2{((0,T);(W^{1,2}_0(\om))^*)}}\leq C_5,
    \qquad
    &
    \norm{(u_\e)_t}_{L^2{((0,T);(W^{1,3}(\om))^*)}}\leq C_5
  \end{align}
for all $\e>0$,
which together with arguments in \cite{Johannes(2016)} 
implies that there exist a sequence 
$(\e_j)_{j\in\mathbb{N}}$ such that $\e_j\searrow 0$ 
as $j\to\infty$ and functions $\n1,\n2,c,u$ such that
  \begin{align*}
  \n1,\n2&\in L^2_{\rm loc}([0,\infty);L^2(\om))\cap 
              L^\frac{4}{3}_{\rm loc}([0,\infty);W^{1,\frac{4}{3}}(\om)),
         \\
        c&\in L^2_{\rm loc}([0,\infty);W^{1,2}(\om)),
         \\
        u&\in L^2_{\rm loc}([0,\infty);W^{1,2}_{0,\sigma}(\om))
  \end{align*}
and that
  \begin{align}
    \notag
    \en1&\to\n1
    &&\mbox{in }L^\frac{4}{3}_{\rm loc}([0,\infty);L^p(\om))
    \quad
    \mbox{for all }p\in\left[1,\frac{12}{5}\right)
    \ \ 
    \mbox{and a.e. in }\om\times(0,\infty),
  \\
    \notag
    \en2&\to\n2
    &&\mbox{in }L^\frac{4}{3}_{\rm loc}([0,\infty);L^p(\om))
    \quad
    \mbox{for all }p\in\left[1,\frac{12}{5}\right)
    \ \ 
    \mbox{and a.e. in }\om\times(0,\infty),
  \\
    \notag
    c_\e&\to c
    &&\mbox{in }C^0_{\rm loc}([0,\infty);L^p(\om))
    \quad
    \mbox{for all }p\in[1,6)
    \ \ \ \,\quad
    \mbox{and a.e. in }\om\times(0,\infty),
  \\
    \notag
    u_\e&\to u
    &&\mbox{in }L^2_{\rm loc}([0,\infty);L^p(\om))
    \quad
    \mbox{for all }p\in[1,6)
    \ \ \ \,\quad
    \mbox{and a.e. in }\om\times(0,\infty),
  \\
    \notag
    c_\e&\to c
    &&\mbox{weakly$^{*}$ in }L^\infty(\om\times(t,t+1))
    \quad
    \mbox{for all }t\geq0,
  \\
    \notag
    \na\en1&\to\na\n1
    &&\mbox{weakly\,\ \ in }L^\frac{4}{3}_{\rm loc}([0,\infty);L^\frac{4}{3}(\om)),
  \\
    \notag
    \na\en2&\to\na\n2
    &&\mbox{weakly\,\ \ in }L^\frac{4}{3}_{\rm loc}([0,\infty);L^\frac{4}{3}(\om)),
  \\
    \notag
    \na c_\e&\to\na c
    &&\mbox{weakly$^{*}$ in }L^\infty_{\rm loc}([0,\infty);L^2(\om)),
  \\
    \notag
    \na u_\e&\to\na u
    &&\mbox{weakly\,\ \ in }L^2_{\rm loc}([0,\infty);L^2(\om)),
  \\
    \notag
    Y_\e u_\e&\to u
    &&\mbox{in }L^2_{\rm loc}([0,\infty);L^2(\om)),
  \\
    \notag
    \en1&\to\n1
    &&\mbox{in }L^2_{\rm loc}([0,\infty);L^2(\om)),
  \\
    \label{prop:2.1}
    \en2&\to\n2
    &&\mbox{in }L^2_{\rm loc}([0,\infty);L^2(\om))
  \end{align}
as $\e=\e_j\searrow0$. Thus we see that $(\n1,\n2,c,u)$ 
is a weak solution to \eqref{P} in the sense of 
Definition \ref{def;weaksol}, which means the end of the proof. \qed
%
%
%
%%==============================================================%%
%%==============                                  ==============%%
%%======                     Section3                     ======%%
%%====                                                      ====%%
%%==                                                          ==%%
%%====                 Proof of Theorem 1.3                 ====%%
%%======                                                  ======%%
%%==============                                  ==============%%
%%==============================================================%%
%
\section{Proof of Theorem 1.3 (Eventual smoothness and stabilization)}\label{section3}
In this section we will prove Theorem \ref{mainthm2}. 
The following lemma plays an important role 
in the proof of Theorem \ref{mainthm2}. 
%\newpage
%%%%%%%%%%%%%%%%%%%%%%%%%%  Lemma 3.1  %%%%%%%%%%%%%%%%%%%%%%%%%%%%
\begin{lemma}\label{lem:3.2}
  \begin{description}
  \item[(i)]
    Assume that $a_1,a_2\in(0,1)$.
    Then there exists $C>0$ such that for all $\e>0$, 
    \begin{align*}
      \int^\infty_0\!\!\!\!\int_\om(\en1-N_1)^2\leq C,
      \quad
      \int^\infty_0\!\!\!\!\int_\om(\en2-N_2)^2\leq C,
    \end{align*}
    where $N_1=\frac{1-a_1}{1-a_1a_2}$, $N_2=\frac{1-a_2}{1-a_1a_2}.$
  \\
  \item[(ii)]
    Assume $a_1\geq a_2>0$. 
    Then there exists $C>0$ such that for all $\e>0$, 
    \begin{align*}
      \int^\infty_0\!\!\!\!\int_\om\en1^2\leq C,
      \quad
      \int^\infty_0\!\!\!\!\int_\om(\en2-1)^2\leq C.
    \end{align*}
  \end{description}
\end{lemma}
%%%%%%%%%%%%%%%%%%%%%%%  Proof of Lemma 3.1  %%%%%%%%%%%%%%%%%%%%%%
\begin{proof}
Due to arguments similar to those in 
\cite[Lemmas 4.1--4.4]{HKMY(2017)}, by using the energy functions
  \begin{align*}
   {\cal G}_{\e,B}:=\int_\om\left(\en1-N_1\log\frac{\en1}{N_1}\right)
                    +\int_\om\left(\en2-N_2\log\frac{\en2}{N_2}\right)
                    +\frac{B}{2}\int_\om c_\e^2
  \end{align*}
in the case that $a_1,a_2\in(0,1)$, and 
  \begin{align*}
    {\cal G}_{\e,B}:=\int_\om\en1
                     +\int_\om\big(\en2-\log\en2\big)
                     +\frac{B}{2}\int_\om c_\e^2
  \end{align*}
in the case that $a_1\geq 1>a_2>0$, 
we can see this lemma.
\end{proof}

%%%%%%%%%%%%%%%%%%%%%%%%%%%%%%%%%%%%%%%%%%%%%%%%%%%%%%%%%%%%%%%%%%%
{\it Proof of Theorem 1.3.}
According to an argument similar to that in the proof of 
\cite[Lemmas 3.4 and 3.5]{Johannes(2016)}, 
for all $\eta>0$ and $p\in(1,\infty)$ there are $T>0$, 
$\e_0>0$ and $C_1>0$ such that for all $t>T$ and $\e\in(0,\e_0)$,
  \begin{align*}
    \norm{c_\e(\cdot,t)}_{L^\infty(\om)}<\eta,
    \quad
    \norm{\en1^p(\cdot,t)}_{L^p(\om)}\leq C_1,
    \quad
    \norm{\en2^p(\cdot,t)}_{L^p(\om)}\leq C_1.
  \end{align*}
We next consider the estimate for $u_\e$.
Since $\na\cdot u_\e=0$, it follows from 
the Young inequality, the Poincar\'{e} inequality, 
boundedness of $\na\Phi$ and \eqref{Pe} 
that there exists $C_2 >0$ such that
\begin{align*}
  \frac{d}{dt}\int_\om|u_\e|^2
      &=-2\int_\om|\na u_\e|^2
        -2\int_\om u_\e\cdot(Y_\e u_\e\cdot\na)u_\e
        +2\int_\om u_\e\cdot(\gamma\en1+\d\en2)\na\Phi
    \\[2mm]
      &=-2\int_\om|\na u_\e|^2
        -2\int_\om u_\e\cdot(Y_\e u_\e\cdot\na)u_\e
    \\
      &\quad\,
       +2\gamma \int_\om u_\e\cdot(\en1-n_{1,\infty})\na\Phi
       +2\d \int_\om u_\e\cdot(\en2-n_{2,\infty})\na\Phi
    \\[2mm]
      &\leq
       -\int_\om|\na u_\e|^2
       -2\int_\om u_\e\cdot(Y_\e u_\e\cdot\na)u_\e
    \\
      &\quad\,
       + C_2 \int_\om(\en1-n_{1,\infty})^2
       + C_2 \int_\om(\en2-n_{2,\infty})^2, 
\end{align*}
where $(n_{1,\infty},n_{2,\infty})=(N_1,N_2)$ 
in the case that $a_1,a_2\in(0,1)$ and 
$(n_{1,\infty},n_{2,\infty})=(0,1)$ in the case that $a_1\geq1>a_2>0$.
Then, noticing from straightforward calculations that 
$\int_\om u_\e\cdot(Y_\e u_\e\cdot\na)u_\e = 
\frac 12 \int_\om \nabla \cdot (Y_\e u_\e) |u_\e|^2 =0$, 
thanks to Lemma \ref{lem:3.2}, 
we obtain from integration of the above inequality 
over $(0,\infty)$ that there exists $C_3>0$ such that
  \begin{align*}
    \int^\infty_0 \int_\om|\na u_\e|^2\leq C_3.
  \end{align*}
According to an argument similar to that in the proof of 
\cite[Lemmas 3.7--3.11]{Johannes(2016)}, 
there exist $\alpha'>0$, $T^*>T$, $C_4>0$ such that
for all $t>T^*$ there exists $\e_1>0$ 
such that for all $\e\in(0,\e_1)$,
$$
\begin{array}{rcl}
  \norm{\en1}_{C^{1+\alpha',\frac{\alpha'}{2}}
                (\overline{\om}\times[t,t+1])}
  \leq C_4,
  &\quad&
  \norm{\en2}_{C^{1+\alpha',\frac{\alpha'}{2}}
                (\overline{\om}\times[t,t+1])}
  \leq C_4,
\\
  \norm{c_\e}_{C^{1+\alpha',\frac{\alpha'}{2}}
                (\overline{\om}\times[t,t+1])}
  \leq C_4,
  &\quad&
  \norm{u_\e}_{C^{1+\alpha',\frac{\alpha'}{2}}
                (\overline{\om}\times[t,t+1])}
  \leq C_4.
\end{array}
$$
Then aided by arguments similar to those in the proofs of 
\cite[Corollary 3.3--Lemma 3.13]{Johannes(2016)},
from \eqref{lem:2.4} there are $\alpha'\in(0,1)$ and $T_0>0$ 
as well as a subsequence $\e_{j}\searrow0$ such that 
for all $t>T_0$
\begin{align*}
  \en1\to\n1,\quad\en2\to\n2,\quad c_\e\to c,\quad u_\e\to u
  \quad\mbox{in}\ 
  C^{1+\alpha',\frac{\alpha'}{2}}(\overline{\om}\times[t,t+1])
\end{align*}
as $\e=\e_j\searrow0$, 
and then 
\begin{align}\label{cor:3.3}
  \notag
  \norm{\n1}_{C^{1+\alpha',\frac{\alpha'}{2}}
               (\overline{\om}\times[t,t+1])}
  \leq C_4,
  &\qquad
  \norm{\n2}_{C^{1+\alpha',\frac{\alpha'}{2}}
               (\overline{\om}\times[t,t+1])}
  \leq C_4,
\\
  \norm{c}_{C^{1+\alpha',\frac{\alpha'}{2}}
             (\overline{\om}\times[t,t+1])}
  \leq C_4,
  &\qquad
  \norm{u}_{C^{1+\alpha',\frac{\alpha'}{2}}
             (\overline{\om}\times[t,t+1])}
  \leq C_4.
\end{align} 
Then we obtain 
\begin{align*}
  \n1,\n2,c,u \in C^{2+\alpha',1+\frac{\alpha'}{2}}
                   (\overline{\om}\times[T_0,\infty)).
\end{align*}
Finally, from \eqref{cor:3.3} the solution $(\n1,\n2,c,u)$ of $\eqref{Pe}$ 
constructed in \eqref{prop:2.1} fulfills
  \begin{align*}
    \n1(\cdot,t)\to N_1,\quad
    \n2(\cdot,t)\to N_2,\quad
    c(\cdot,t)\to 0,\quad
    u(\cdot,t)\to 0\quad
    \mbox{in}\ C^1(\overline{\om})\quad
    (t\rightarrow\infty)
  \end{align*}
in the case that $a_1,a_2\in(0,1)$, and
  \begin{align*}
    \n1(\cdot,t)\to 0,\quad
    \n2(\cdot,t)\to 1,\quad
    c(\cdot,t)\to 0,\quad
    u(\cdot,t)\to 0,\quad
    \mbox{in}\ C^1(\overline{\om})\quad
    (t\rightarrow\infty)
  \end{align*}
in the case that $a_1\geq 1>a_2>0$,
which enable us to see Theorem \ref{mainthm2}. \qed 

%==========================================================
%%%%%%%                                             %%%%%%%
  %%%                                                 %%%
 %%%                                                   %%%
%%%                    Acknowledgments                  %%%
 %%%                                                   %%%
  %%%                                                 %%%
%%%%%%%                                             %%%%%%%
%==========================================================

%%==============================================================%%
%%==============                                  ==============%%
%%======                                                  ======%%
%%====                                                      ====%%
%%==                         Reference                        ==%%
%%======                                                  ======%%
%%==============                                  ==============%%
%%==============================================================%%
% \newpage
 

\begin{thebibliography}{99}
  \bibitem{BBTW}
    {\sc N. Bellomo, A. Bellouquid, Y. Tao, and M. Winkler}, 
    {\it Toward a mathematical theory 
    of Keller--Segel models 
    of pattern formation in biological tissues}, 
    Math.\ Models Methods Appl.\ Sci., 25 
     (2015), pp. 1663--1763.
  \bibitem{CKM1}
     {\sc X. Cao, S. Kurima, and M. Mizukami}, 
     {\it Global existence and asymptotic behavior 
     of classical solutions for a 3D two-species 
     chemotaxis-Stokes system with competitive kinetics}, 
     	arXiv: 1703.01794 [math.AP]. 
  \bibitem{CKM2}
     {\sc X. Cao, S. Kurima, and M. Mizukami}, 
     {\it Global existence and asymptotic behavior 
     of classical solutions for a 3D two-species 
     Keller--Segel-Stokes system with competitive kinetics}, 
     	arXiv: 1706.07910 [math.AP]. 
  \bibitem{HKMY(2017)}
     {\sc M. Hirata, S. Kurima, M. Mizukami, and T. Yokota}, 
     {\it Boundedness and stabilization in a 
     two-dimensional two-species chemotaxis-Navier--Stokes system 
     with competitive kinetics}, 
     J. Differential Equations, 263 (2017), pp. 470--490.
  \bibitem{Johannes(2016)}
    {\sc J. Lankeit}, 
    {\it Long-term behaviour in a chemotaxis-fluid system 
    with logistic source}, 
    Math.\ Models Methods Appl.\ Sci., 26 (2016), 
    pp. 2071--2109.
  \bibitem{TW}
    {\sc Y. Tao and M. Winkler}, 
    {\it Blow-up prevention by quadratic degradation 
    in a two-dimensional Keller--Segel-Navier--Stokes system}, 
    Z. Angew.\ Math.\ Phys., 67 (2016), 
    23 pp.
 \end{thebibliography}
\end{document}